\documentclass[11pt,reqno,a4paper]{amsart}
\usepackage[latin1]{inputenc}
\usepackage[latin1]{inputenc}
\usepackage{a4wide}
\usepackage{fancyhdr}
\usepackage{mathtools}
\usepackage{amsfonts,amssymb,amsmath,amsthm,latexsym,indentfirst}
\usepackage{epsfig}
\usepackage{enumerate}
\usepackage{amsfonts}
\usepackage{latexsym}
\usepackage{amssymb}
\usepackage{amsmath}
\usepackage{textcomp}
\usepackage{graphicx}
\usepackage[bookmarks=false,colorlinks=true,linkcolor={blue},pdfstartview={XYZ null null 1.22}]{hyperref}
\usepackage{hyperref}

\numberwithin{equation}{section}

\newtheorem{theorem}{Theorem}[section]
\newtheorem{definition}[theorem]{Definition}
\newtheorem{lemma}[theorem]{Lemma}
\newtheorem{proposition}[theorem]{Proposition}
\newtheorem{remark}[theorem]{Remark}

\newcommand{\N}[1][]{\ensuremath{{\mathbb{N}^{#1}} }}

\newcommand{\R}[1][]{\ensuremath{{\mathbb{R}^{#1}} }}

\newcommand{\re}{\mathbb{R}}

\newcommand{\z}{\mathcal{Z}}

\newcommand{\E}{\mathcal{E}}

\title[On the nonlinear Schr\"odinger equation]{On the nonlinear Schr\"odinger equation in spaces of infinite mass and low regularity}

\author[Vanessa Barros, Sim\~ao Correia and Filipe Oliveira ]{Vanessa Barros, Sim\~ao Correia and Filipe Oliveira }

\keywords{Nonlinear Schrödinger equation; local well-posedness}

\subjclass[2010]{35Q55; 35Q60; 35A01; 35A02; 35B40; 35B06; 35A23; 35B30; 78A45}

\begin{document}

\maketitle

\begin{abstract}
We study the  nonlinear Schr\"odinger equation with initial data in $\z^s_p(\re^d)=\dot{H}^s(\R^d)\cap L^p(\R^d)$, where $0<s<\min\{d/2,1\}$ and $2<p<2d/(d-2s)$. After showing that the linear Schrödinger group is well-defined in this space, we prove local well-posedness in the whole range of parameters $s$ and $p$. The precise properties of the solution depend on the relation between the power of the nonlinearity and the integrability $p$. Finally, we present a global existence result for the defocusing cubic equation in dimension three for initial data with infinite mass and energy, using a variant of the Fourier truncation method.
\end{abstract}

\section{Introduction}\label{intro}
\noindent
It is well-known that the Schr\"odinger equations
\begin{equation}
\label{eqini}
iu_t+\Delta u+f(|u|)u=0,
\end{equation}
where $f\,:\,\re^+\to \re$, admit the formal mass-invariant
$$M(u)=\int |u|^2dx.$$ 
For initial data in $L^2$-based spaces, such as the Sobolev spaces $H^s(\re^d)$, the local well-posedness of \eqref{eqini} and the problem of  global existence versus blow-up in finite time of solutions (in which the mass of the initial data often plays a very important role) are now relatively well-understood for large classes of nonlinear potentials $f(|u|)$. On the other hand, for initial data $u_0\not\in L^2(\re^d)$, which we refer to as having ``infinite mass", results in the literature are rather scarce. One of the first works in this direction is due to Zhidkov: in \cite{zk}, the author studied the well-posedness of the Gross-Pitaevskii equation
\begin{equation}\tag{GP}
\label{gp}
iu_t+\Delta u-(|1-|u|^2)u=0
\end{equation}
in the space
$$X^1(\re)=L^{\infty}(\re)\cap\dot{H}^1(\re).$$
This space is specially suited to treat cases where \eqref{gp} is complemented with a non-zero boundary condition at infinity ($u\to \pm 1$). This situation arises naturally in several physical contexts described by \eqref{gp}, such as the study of dark solitons in nonlinear optics (\cite{kivshar}) or the propagation of solitary waves in Bose condensates (see for instance \cite{sulem} and the references therein).\\
The results in \cite{zk} were extended in \cite{gallo}  (see also \cite{goubet}) to higher spatial dimensions through the following generalizations of $X^1(\re)$, coined as ``Zhidkov spaces":
$$X^k(\re^d)=\{u\in L^{\infty}(\re^d)\,:\,\nabla u\in H^{k-1}(\re^d)\}.$$

\bigskip

\noindent
Also, in \cite{BV08}, the nonlinear Schr\"odinger equation 
\begin{equation}\label{NLS}\tag{NLS}
iu_t+\Delta u+ \lambda|u|^{\sigma}u=0,
\end{equation}
where $\lambda=\pm 1$, was considered with initial data of infinite-mass $u_0=a\delta_0+\tilde{u}_0$ for $a\in\re\setminus \{0\}$ and $\tilde{u}$ a regular function. In particular, local well-posedness was proved for $\tilde{u}_0\in L^2(\re^d)$ and $\sigma<2/d$.

\medskip
%

\noindent
In \cite{Correia}, we showed the local well-posedness of \eqref{NLS} for initial data in 
$$\z^1_p(\re^d)\coloneqq L^p(\re^d)\cap \dot{H}^1(\re^d) $$ for $2<p\le2\sigma+2$ and $d>2,$ and for $p>2\sigma+2$ in dimensions $d=1$ and $d=2$. Note also that, in \cite{gallo}, the well-posedness of \eqref{NLS} was addressed in $X^k(\re^d)$ for $s=k$ integer, $k>d/2$. Under this assumption, the space $X^k$ is an algebra and the local existence theory follows from standard arguments.
In the present work, we 
extend these results to rough initial data in
 $$\z_p^s(\re^d):= L^p(\re^d)\cap \dot{H}^s(\re^d),$$
 under the conditions
 \begin{equation}
 \label{condf}
 0<s< \min\{d/2, 1\},\,\quad 2<p<2^*:=\frac{2d}{d-2s}.
 \end{equation}
 The first essential step is the solvability of the linear equation in $\z^s_p$:
 \begin{theorem}\label{teo:linear}
 	Let $s$ and $p$ as in \eqref{condf}. The Schr\"odinger group $\{e^{it\Delta}\}_{t\in\re}$ acts continuously on 
 	$\z^s_p(\re^d)$. Moreover, for $u_0\in \z^s_p(\re^d)$, 
 	\begin{equation}\label{eq:estimativa_grupo}
 	\|e^{it\Delta}u_0\|_{\z^s_p}\lesssim (1+|t|)\|u_0\|_{\z^s_p}.
 	\end{equation}
 \end{theorem}
 
 We now state the local existence results for \eqref{NLS}. To clarify, we say that $u$ is a solution of \eqref{NLS} with initial condition $u_0$ if it satisfies the corresponding Duhamel formula
 \begin{equation}\label{eq:duhamel}
 u(t) = e^{it\Delta} u_0 +i\lambda \int_{0}^{t}e^{i(t-\tau)\Delta} |u(\tau)|^{\sigma}u(\tau)d\tau.
 \end{equation}
 When we apply a fixed-point argument to \eqref{eq:duhamel}, the main difficulty is the control of the $L^\infty_tL^p_x$ norm, which is not in the scope of the usual Strichartz estimates. There are two main approaches to tackle this problem. The first consists in proving that the integral term in the Duhamel formula, $u(t)-e^{it\Delta}u_0$, is in $L^q_tL^r_x$, for an admissible pair $(q,r)$ (in the Strichartz sense \eqref{eq:admissivel}). However, as in \cite{Correia}, this strategy only works if $p\le 2\sigma+2$:
 \begin{theorem}\label{teo:exist_p_peq}
 	Let $s$ and $p$ as in \eqref{condf}, $p\le 2\sigma+2$ and $u_0\in \z^s_p(\R^d)$. Then there exist $T=T(\|u_0\|_{\z^s_p})>0$, an admissible pair $(q,r)$ and a unique solution
 	$$u\in C([0,T],\z^s_p(\R^d))\cap L^q((0,T),\dot{W}^{s,r}(\R^d))$$
 	to \eqref{NLS} with initial condition $u(0)=u_0$. Moreover
 	\begin{equation}\label{eq:ganho_intg}
 	u-e^{it\Delta}u_0\in C([0,T],L^2(\R^d))\cap L^q((0,T),L^r(\R^d)).
 	\end{equation}
 \end{theorem}
 Observe that \eqref{eq:ganho_intg} represents a gain of integrability through the nonlinear term. In other words, the large tails of the solution evolve linearly, while the nonlinear effects are localized in space. As it will be seen, this property is a direct consequence of the classical Strichartz estimates, which are in themselves an embodiment of gain in integrability. For $p>2\sigma +2$, this feature is no longer available and we follow the second approach, which consists in using directly \eqref{eq:estimativa_grupo} and then reach an admissible $L^q_tL^r_x$ at the expense of $s$ derivatives. 
 \begin{theorem}\label{teo:exist_p_grand}
 	Let $s$ and $p$ as in \eqref{condf}, $p> 2\sigma+2$, and $u_0\in \z^s_p(\R^d)$. Then there exist  $T=T(\|u_0\|_{\z^s_p})>0$,  admissible pairs $(q_j,r_j),\ j=1,2,3$, and a unique solution
 	$$u\in C([0,T],\z^s_p(\R^d))\bigcap_{j=1}^3 L^{q_j}((0,T),\dot{W}^{s,r_j}(\R^d))$$
 	to \eqref{NLS} with initial condition $u(0)=u_0$.
 \end{theorem}
\begin{remark}
	In fact, the proof of this theorem naturally applies to the case $s=1$, thus extending the local well-posedness theory in $\z_p^1(\re^d)$ developed in \cite{simaoLpH1}.
\end{remark}

 Finally, we consider the global existence problem. When compared with the classical theory in $H^s(\re^d)$ spaces, two main difficulties arise: first, the linear group is not bounded in $\z^s_p(\re^d)$ uniformly in time, which means that global results for small data are not straightforward. Second, the conservation laws for the mass $M$ and the energy
 $$
 E(u)=\frac{1}{2}\int_{\re^d} |\nabla u|^2 dx -\frac{\lambda}{\sigma+2}\int_{\re^d} |u|^{\sigma+2} dx
 $$
 are not available in $\z_s^p(\re^d)$. In \cite{simaoLpH1}, a global result in $\z_{\sigma+2}^1(\re^d)$ was obtained for $\lambda=1$ and $\sigma$ small: even though the mass is not an available quantity, one may still derive a corrected mass estimate by throwing away the large tails of the solution. Since the energy is well-defined in $\z_{\sigma+2}^1(\re^d)$ (and thus conserved), this estimate is sufficient to apply a Gagliardo-Nirenberg argument. The case $\lambda=-1$ is trivial, as the energy provides a direct bound for the $\z_{\sigma+2}^1(\re^d)$ norm. It is interesting to notice that, in this case, the nonlinear evolution satisfies stronger bounds than the linear one, where a growth of the $L^{\sigma+2}(\re^d)$ norm is observed.
 
 Our goal is to understand whether a global existence result is possible when \textit{neither mass nor energy are well-defined}. To that end, we sought inspiration in the Fourier truncation method developed by Bourgain (\cite{bourgain2}, \cite{bourgain}) for the defocusing case $\lambda=-1$, where the a priori bound provided by the energy conservation, together with a nonlinear smoothing effect, allows one to extend the global existence result from $H^1(\re^d)$ to $H^s(\re^d)$, for $s$ close to 1. The technique has been refined in various papers (\cite{ckstt_dnls}, \cite{ckstt_kdv}, \cite{ckstt}), where one employs several multilinear estimates on Bourgain spaces. Since these refinements deeply rely on the conservation of the mass, we have decided to follow the more robust arguments of Bourgain instead. 
 
 In our context, one may consider the problem with either a lack of regularity (that is, in $\z_{\sigma+2}^s(\re^d)$, $s<1$), or with a lack of integrability (in $\z_p^1(\re^d)$, $p>\sigma+2$). The first direction seems quite similar to the $H^s(\re^d)$ case (but not entirely trivial). We therefore focus on the second possibility. Following the heuristics of the Fourier truncation method, the key property is a nonlinear ``smoothing'' estimate which, in the context of integrability, is akin to \eqref{eq:ganho_intg}. This argument, which suggests that global existence should be valid for $p$ close to $\sigma +2$, is made concrete in the following particular case:
 
 \begin{theorem}\label{teo:existglobal}
 	Set $\sigma=2$ and $\lambda=-1$. Then \eqref{NLS} is globally well-posed in $\z^1_p(\re^3)$, for any $4<p<9/2$.
 \end{theorem}
 
 The methodology is in fact quite general and similar results should be achievable in the general case (at least for $\sigma$ even). This should be pursued in future works. 
\bigskip

\noindent
The rest of this paper is organized as follows: in Section \ref{notation} we introduce some notations and study the linear problem. In Sections \ref{mainsection} and \ref{main2} we prove the local well-posedness of $\eqref{NLS}$ for $p\le 2\sigma+2$ and $p>2\sigma +2$, respectively. Finally, in Section \ref{sec:existglobal}, we prove Theorem \ref{teo:existglobal}.
\section{Preliminaries and linear estimates}\label{notation}
We denote the Fourier transform by $\mathcal{F}$. In our analysis, we will use various Sobolev and Besov spaces, which we now recall for the sake of convenience:
\begin{itemize}
	\item the homogeneous Sobolev space over $L^r(\re^d)$ of order $s\in\re$,
	$$\dot{W}^{s,r}(\re^d)=\{f\in S' (\re^d)/\mathcal{C}(\re^d)\,:\,\mathcal{F}^{-1}(|\xi|^s\mathcal{F}f) \in L^r(\re^d) \},$$
	where $\mathcal{C}(\re^d)$ is the set of all constant distributions;
	\item the classical inhomogeneous Sobolev space of order $s$,
	$$W^{s,r}(\re^d)=\{f\in S' (\re^d)\,:\,\mathcal{F}^{-1}((1+|\xi|^2)^{s/2}\mathcal{F}f) \in L^r(\re^d) \};$$
	\item for $s \in \re$, $1\le p,q\le\infty$, the homogeneous Besov spaces,
	$$
	{\dot{B}}^{s,q}_{p}(\re^d)\coloneqq \left\{  f\in \mathcal{S}'(\mathbb{R})/\mathcal{C}(\re^d); \;
	\|f\|_{{B}^{s,q}_{p}}=\big{(}\sum_{j\in \mathbb{Z}}2^{qjs}\|\Delta_j f\|_{L^p}^q\big{)}^{\frac{1}{q}}
	<+\infty\,\right\},
	$$
	where $\Delta_j$ is the Littlewood-Paley projection onto frequencies $\sim 2^j$;
	\item the corresponding inhomogeneous Besov spaces,
	$$
	{B}^{s,q}_{p}(\re^d)\coloneqq \left\{  f\in \mathcal{S}'(\mathbb{R}^d)\,:\,
	\|f\|_{{B}^{s,q}_{p}}=\|S_0 f\|_{L^p}+
	\big{(}\sum_{j\geq 1}2^{qjs}\|\Delta_j f\|_{L^p}^q\big{)}^{\frac{1}{q}}
	<+\infty\,\right\},
	$$
	where $S_0$ is the projection onto frequencies $\lesssim 1$.
\end{itemize}

\bigskip
\noindent
In the remainder of this section, we study the linear problem for initial data $u_0\in\z_p^s(\re^d)$. 
By the Sobolev injection, we have
\begin{equation}
\label{injecoes}
H^s(\re^d)\hookrightarrow\z^s_p(\re^d)\hookrightarrow L^{q}(\re^d),\quad p\le q\le 2^*.
\end{equation}

\noindent
Following the ideas in \cite{gallo}, we introduce the following space:
\begin{definition}
	Given $n\in \N$, $s>0$ and $1\le p\le +\infty$, we define the function space $$\z^{[s,s+n]}_p(\re^d)\coloneqq \underset{\alpha\in [s,s+n]}{\bigcap} \z^{\alpha}_p(\re^d)$$
	endowed with the norm
	$$\|\cdot\|_{\z^{[s,s+n]}_p}=\underset{\alpha\in [s,s+n]}{\sup}\|\cdot\|_{\z_p^\alpha}$$
\end{definition}
\noindent
The next result states that any element in $\z_p^s$ can be canonically decomposed into a regular slowly decaying part and a rough localized term.

\begin{lemma}\label{gallo}
 Let $n\in \N$ and $2\le p\le 2^*$. Then  $$\z^s_p(\re^d)\cong H^s(\re^d)+ \z^{[s,s+n]}_p(\re^d).$$ 
\end{lemma}

\medskip

\begin{proof}
 Given $g\in \z^s_p(\re^d)$, we define $g_1=\check\psi\ast g,$ where $\psi$ is a bump function localized in the unit ball.  We claim  that $g_1\in \z^{[s,s+n]}_p(\re^d)$. In fact taking $\alpha\in [s,s+n]$ and writing $\alpha=s+k,\ 0\le k\le n$, we have that
 $$\|D^\alpha g_1\|_{L^2}=\|D^k \check\psi \ast D^s g\|_{L^2}\lesssim \| g\|_{\dot{H}^s}$$
 and
  $$\|g_1\|_{L^p}\lesssim \| g\|_{L^p}$$
by Young's inequality and the fact that $\psi\in \mathcal{S}(\re^d).$ 
  Therefore 
 \begin{equation}\label{gallo1}
 \|g_1\|_{\z^{[s,s+n]}_p}=\underset{\alpha\in [s,s+n]}{\sup}\|g_1\|_{\z_p^\alpha}\lesssim \|g\|_{\z_p^s}.
 \end{equation}
 It remains to prove that $g_2\coloneqq  g-g_1\in H^s(\re^d).$ In fact
  \begin{equation*}
 \|g_2\|_{H^s}=\|J^s\hat{g}(1-\psi)\|_{L^2}\lesssim \|\hat{g}(1-\psi)\|_{L^2}+\||\xi|^s\hat{g}(1-\psi)\|_{L^2}.
 \end{equation*}
 Observing that $\frac{1-\psi}{|\xi|^s}$ is bounded,
   \begin{equation}\label{gallo2}
 \|g_2\|_{H^s}\le \||\xi|^s\hat{g}\|_{L^2}\Big\|\frac{(1-\psi)}{|\xi|^s}\Big\|_{L^\infty}+\||\xi|^s\hat{g}\|_{L^2}\|{1-\psi}\|_{L^\infty}\lesssim \|g\|_{\z_p^s}.
 \end{equation}
 \noindent
 Finally, the embedding $$H^s(\re^d)+ \z^{[s,s+n]}_p(\re^d)\hookrightarrow \z^s_p(\re^d)$$ is
a direct consequence of \eqref{injecoes}.
\end{proof}

\begin{proof}[Proof of Theorem \ref{teo:linear}]
	 Let  $u_0\in {\z^s_p}(\re^d)$. From Lemma \ref{gallo}, taking $n=2$, we can decompose $u_0$ as $u_0=v_0+w_0,\,\textrm{ with } v_0\in H^s(\re^d) \text{ and } w_0\in \z^{[s,s+2]}_p(\re^d).$ We set $u(t)=e^{it\Delta}u_0$ and introduce $v(t)=u(t)-w_0.$ Then $v$ satisfies the linear IVP
	 \begin{equation}\label{linearv}
	 \left\{\begin{array}{l} i \partial_t v + \Delta v  = -\Delta w_0, \\
	 v(x,0)  = v_0(x)
	 \end{array}
	 \right.\quad  (x, t) \in \re^d\times \re,
	 \end{equation}
and
	 \begin{equation}
	 \label{eqv}
	 v(t)=e^{it\Delta}v_0-i\int_0^t e^{i(t-\tau)\Delta}(\Delta w_0) d\tau.
	 \end{equation}
	 Notice that $v(t)\in H^s(\re^d)$. Indeed,
	 \begin{align*}
	 \|v(t)\|_{H^s}&\le\|e^{it\Delta}v_0\|_{H^s}+\int_0^t \|e^{i(t-\tau)\Delta}(\Delta w_0)\|_{H^s}d\tau\le \|v_0\|_{H^s}+|t|(\|w_0\|_{\dot{H}^2}+\|w_0\|_{\dot{H}^{s+2}})\\&\le \|v_0\|_{H^s}+2|t|\|w_0\|_{\z_p^{[s;s+2]}}
	 \end{align*}
	 since $s\le 2$. Hence, by Lemma \ref{gallo},
	 $u(t)=v(t)+w_0\in H^s(\re^d)+\z^{[s,s+2]}_p(\re^d)\hookrightarrow \z_p^s(\re^d)$
	 and 
	 $$\|u(t)\|_{\z_p^s}\le \|v(t)\|_{H^s}+ \| w_0\|_{\z^{[s,s+2]}}=\|v_0\|_{H^s}+2|t|\|w_0\|_{\z_p^{[s;s+2]}}+\| w_0\|_{\z^{[s,s+2]}}\lesssim (1+|t|)\|u_0\|_{\z^s_p}.$$

\noindent
Next, we show that $\{e^{it\Delta}\}_{t\in\re}$ defines a continuous group over $\z_s^p(\re^d)$.
From  \eqref{eq:estimativa_grupo}, 
 $$e^{it\Delta}:\z^s_p(\re^d)\mapsto \z^s_p(\re^d)$$ is a bounded operator and we only need to show, for $\phi\in \z_p^s(\re^d)$, that
 \begin{equation}
 \underset{t\rightarrow 0}{\lim}\|e^{it\Delta}\phi-\phi\|_{\z^s_p}=0.\label{simao3}
 \end{equation}

\noindent Consider a sequence $(\phi_n) \in H^s(\re^d)$ such that $\phi_n\underset{\z_p^s}{\rightarrow} \phi.$ Then

 $$\|e^{it\Delta}\phi-\phi\|_{\z^s_p}\le \|e^{it\Delta}(\phi-\phi_n)\|_{\z^s_p} +\|e^{it\Delta}\phi_n-\phi_n\|_{\z^s_p}+\|\phi_n-\phi\|_{\z^s_p}.$$
Given $\epsilon>0$, we choose $n$ such that  $\|\phi_n-\phi\|_{\z^s_p}<\epsilon$ and for any $|t|<1$, $\|e^{it\Delta}(\phi-\phi_n)\|_{\z^s_p}\lesssim\epsilon$ (by \eqref{eq:estimativa_grupo}). By continuity of $\{e^{it\Delta}\}_{t\in \R}$ over $H^s(\re^d)$, for $t$ small enough, $\|e^{it\Delta}\phi-\phi\|_{\z^s_p}\lesssim\|e^{it\Delta}\phi_n-\phi_n\|_{H^s}+ 2\epsilon \lesssim 3\epsilon$.
 \end{proof}

\section{Local well-posedness for $p\le 2\sigma+2$}\label{mainsection}
\noindent
We begin this section by stating some Strichartz estimates in Besov spaces. For the proof of these results  the reader may refer to \cite{cazenave}. As usual, we say the pair $(q,r)$ is admissible if 
\begin{equation}\label{eq:admissivel}
\dfrac 2q=d\Big(\dfrac{1}{2}-\dfrac{1}{r}\Big),\ 2\le r\le \left\{\begin{array}{ll}
2d/(d-2),& d\ge 3\\
\infty,& d=1,2.
\end{array}\right. \mbox{ and, if }d=2, r\neq \infty.
\end{equation}
\begin{lemma}\label{stric}
	Given $s\in \R$ and $(q,r)$, $(\gamma,\rho)$ admissible pairs, the following properties hold:
	\begin{enumerate}
		\item {\bf Homogeneous estimate:} 
		$$\forall u_0 \in H^s(\re^d),\,\|e^{it\Delta}u_0\|_{L^q(\R,B^{s,2}_r)}\lesssim \|u_0\|_{B^{s,2}_2}\sim \|u_0\|_{H^s}.$$ 
		\item {\bf Inhomogeneous estimate:} Let $I$ be an interval of $\R$ (bounded or not) and $t_0\in\bar I$. There exists a constant $C$ independent of $I$ (and of $t_0\in I$) such that
		$$\forall g \in {L^{\gamma'}(I,B^{s,2}_{\rho'}(\re^d))},\,\Big\|\int_{t_0}^te^{i(t-\tau)\Delta}g(\cdot,\tau)d\tau\Big\|_{L^q(I,B^{s,2}_r)}\le C \|g\|_{L^{\gamma'}(I,B^{s,2}_{\rho'})}.$$
			\end{enumerate}
Both estimates remain valid by replacing the Besov spaces by their homogeneous counterpart.
\end{lemma}

\bigskip

\noindent
Throughout this section, we assume that \eqref{condf} holds and that $(q,r)$ is an admissible pair
 that will be fixed later.
Before proving Theorem \ref{teo:exist_p_peq}, let us define the function spaces where the solution of \eqref{NLS} will be obtained. Define, for $T>0$,
$$S_0\coloneqq  L^\infty((0,T),L^2(\re^d))\cap L^q((0,T),L^r(\re^d)),$$
$$S_s\coloneqq  L^\infty((0,T),H^s(\re^d))\cap L^q((0,T),{B}^{s,2}_r(\re^d))\hookrightarrow S_0$$
and, given $u_0 \in \z^s_p(\R^d)$ and $M>0$,
\begin{align}\label{defE}
\E\coloneqq  \{&u\in L^\infty((0,T),\z^s_p(\re^d)) \cap L^q((0,T),\dot{B}^{s,2}_r(\re^d)) \,:\,\nonumber\\
 &\|u\|_{\E}\coloneqq \|u\|_{L^\infty((0,T),L^p)}+\|u-e^{it\Delta}u_0\|_{S_s}\le M\}.
\end{align}

\begin{proposition}
	\label{ms} Let $u_0\in \z^s_p(\R^d)$ and $M>0$. Setting $d(u,v):=\|u-v\|_{S_0}$,  $(\mathcal{E},d)$ is a complete metric space.
\end{proposition}
\begin{proof}
Let us first notice that for $u,v\in \E,$ 
$$\|u-v\|_{S_0}\le \|u-e^{it\Delta}u_0\|_{S_0}+\|v-e^{it\Delta}u_0\|_{S_0}\lesssim \|u-e^{it\Delta}u_0\|_{S_s}+\|v-e^{it\Delta}u_0\|_{S_s}\le 2M$$ 
and $d$ is well-defined over $\E$.

\noindent
Consider a Cauchy sequence $(u_n)_{n\in\N}$ in $\E$. Since 
$$v_n:=u_n-e^{it\Delta}u_0\in S_s\hookrightarrow S_0$$
and, for fixed $\epsilon>0$,
$$\|v_m-v_n\|_{S_0}=d(u_m,u_n)<\epsilon \text{ for $n$ and $m$ large enough,}$$
it follows that $(v_n)_{n\in\N}$ is a Cauchy sequence in $S_0$. Thus there exists  $v\in S_0$ such that
\begin{equation}\label{a}v_n=u_n-e^{it\Delta}u_0 \underset{S_0}{\rightarrow} v.\end{equation}
As $(v_n)_{n\in\N}$ is bounded in $S_s$, applying \cite[Theorem 1.2.5]{cazenave}, we see that
$v\in S_s$
and 
\begin{equation}\label{b}\|v\|_{S_s} \le \liminf \|u_n-e^{it\Delta}u_0\|_{S_s}.\end{equation}
Since $u_0 \in \z^s_p(\R^d)\hookrightarrow\dot{H}^{s}(\re^d)$, we have from Theorem \ref{teo:linear} and Lemma \ref{stric} that 
$$e^{it\Delta}u_0 \in L^\infty((0,T),\z^s_p(\R^d))\cap L^q((0,T),\dot{B}^{s,2}_r(\re^d)).$$
Hence, setting $u=v+e^{it\Delta}u_0$, we conclude that $d(u_n,u)\to 0$ and $u\in L^\infty((0,T),\z^s_p(\R^d))\cap L^q((0,T),\dot{B}^{s,2}_r(\re^d))$. Finally, for each $t\in (0,T)$, since $u_n(t)-u(t)\to 0$ in $L^2(\re^d)$, $u_n(t)\rightharpoonup u(t)$ in $L^p(\re^d)$ and
$$
\|u\|_{\E}= \|u\|_{L^\infty((0,T), L^p)} + \|v\|_{S_s}\le \liminf \left(\|u_n\|_{L^\infty((0,T), L^p)} + \|v_n\|_{S_s} \right)\le M.
$$
\end{proof}
%
\begin{proof}[Proof of Theorem \ref{teo:exist_p_peq}]
Consider the integral operator 
$$\Phi[u](t)\coloneqq e^{it\Delta}u_0+i\lambda\int_0^te^{i(t-\tau)\Delta}|u(\tau)|^\sigma u(\tau)d\tau,\ 0<t<T, u\in\mathcal{E}.$$
We  show that $\Phi$ is a contraction in $(\E,d)$. We begin by estimating
\begin{align}\label{ja}
\|\Phi[u](t)-e^{it\Delta}u_0\|_{S_s}=\Big\|\lambda\int_0^te^{i(t-\tau)\Delta}|u(\tau)|^\sigma u(\tau)d\tau\Big\|_{S_s}.
\end{align}
Applying Strichartz estimates,
\begin{align*}
\|\Phi[u](t)-e^{it\Delta}u_0\|_{S_s}\lesssim\||u|^\sigma u\|_{L^{\gamma'}((0,T),{B}^{s,2}_{\rho'})}.
\end{align*}
\noindent {\bf Case 1: $p\geq 2d(\sigma+1)/(d+2)$ or $d=1,2$.}\\
We put $r=p$ and define $\rho'=r/(\sigma+1)$. Observe that, if $d\ge 3$, $\rho\le 2d/(d-2)$ is equivalent to $p\geq 2d(\sigma+1)/(d+2)$. Using \cite[Proposition 4.9.4]{cazenave},  we infer that
\begin{align}\label{compare}
\||u|^\sigma u\|_{{B}^{s,2}_{\rho'}}\lesssim \| u\|_{L^a}^\sigma\| u\|_{{B}^{s,2}_{r}},
\end{align}
where $\displaystyle a=\sigma\rho'r/(r-\rho')=r=p$. Using Lemma \ref{stric}, this estimate implies that, for $(\gamma, \rho)$ admissible,
$$
\|\Phi[u](t)-e^{it\Delta}u_0\|_{S_s} \lesssim \left\| \|u\|_{L^a}^\sigma \|u\|_{B^{s,2}_r}\right\|_{L^{\gamma'}(0,T)}.
$$
We now invoke the embedding \eqref{injecoes} and Hölder's inequality to deduce that 
\begin{align}\label{i}
\|\Phi[u](t)-e^{it\Delta}u_0\|_{S_s}\lesssim & \| u\|_{L^{\infty}((0,T),{\z}^{s}_{p})}^\sigma \| u\|_{L^{\gamma'}((0,T),{B}^{s,2}_{r})}\\
\lesssim &\| u\|_{L^{\infty}((0,T),{\z}^{s}_{p})}^\sigma T^{\frac{1}{\gamma'}-\frac{1}{q}}\| u\|_{L^{q}((0,T),{B}^{s,2}_{r})}.\nonumber
\end{align}
Furthermore,
\begin{align*}
\| u\|_{{B}^{s,2}_{r}}\lesssim &\| u\|_{L^{r}}+\| u\|_{\dot{B}^{s,2}_{r}}\lesssim  \| u\|_{\z^s_p}+\| u\|_{\dot{B}^{s,2}_{r}}\\
\le &  \| u\|_{\z^s_p}+\| u-e^{it\Delta}u_0\|_{\dot{B}^{s,2}_{r}}+\| e^{it\Delta}u_0\|_{\dot{B}^{s,2}_{r}}.
\end{align*}

\noindent
From Lemma \ref{stric},
\begin{equation}\label{eq:controlobesov}
\| u\|_{L^{q}((0,T),{B}^{s,2}_{r})}\le T^{1/q}\|u\|_{L^\infty((0,T), \z_p^s)} +\| u-e^{it\Delta}u_0\|_{S_s}+\| u_0\|_{\dot{H}^{s}}.
\end{equation}
Observing that
\begin{align}\| u\|_{L^{\infty}((0,T),{\z}^{s}_{p})}&\le \| u\|_{L^{\infty}((0,T),L^{p})}+\| u\|_{L^{\infty}((0,T),\dot{H}^s)}\nonumber
\\\label{minh}
&\le \| u\|_{L^{\infty}((0,T),L^{p})}+\| u-e^{it\Delta}u_0\|_{L^{\infty}((0,T),\dot{H}^s)}+T\|u_0\|_{\dot{H}^s}\\&\lesssim M + T\|u_0\|_{\z_p^s},\nonumber
\end{align}
equation \eqref{eq:controlobesov} implies
\begin{equation}\label{eq:controlobesov2}
\| u\|_{L^{q}((0,T),{B}^{s,2}_{r})}\lesssim T^{1/q}\left(M + T\|u_0\|_{\z_p^s}\right) + M +\| u_0\|_{\z_p^s}.
\end{equation}
Therefore, by inequalities \eqref{i}, \eqref{minh} and \eqref{eq:controlobesov2}, we conclude that
\begin{equation}\label{iiii}
\|\Phi(u)-e^{it\Delta}u_0\|_{S_s}\lesssim
 (M+T\|u_0\|_{\z_p^s})^\sigma T^{\frac{1}{\gamma'}-\frac{1}{q}}(T^{1/q}(M+T\|u_o\|_{\z_p^s})+M+\| u_0\|_{\dot{H}^{s}}).
\end{equation}

\bigskip

\noindent
Next, using Theorem \ref{teo:linear} and the Sobolev embedding, we infer that
 \begin{align*}
\|\Phi(u)\|_{L^{\infty}((0,T),L^p)}&\le \|e^{it\Delta}u_0\|_{L^{\infty}((0,T),\z^s_{p})}+\|\Phi(u)-e^{it\Delta}u_0\|_{L^{\infty}((0,T),H^s)}\\
&\le (1+T)\|u_0\|_{\z^s_{p}}+\|\Phi(u)-e^{it\Delta}u_0\|_{S_s}.
\end{align*}
Choosing $M=2\|u_0\|_{\z^s_{p}}$ and $T>0$ sufficiently small (depending exclusively on $\|u_0\|_{\z^s_{p}}$), we know that 
$$\Phi(u)\in C([0,T],\z^s_p(\re^d)) \cap L^q((0,T),\dot{B}^{s,2}_r(\re^d)) \hookrightarrow \E\mbox{ and }\|\Phi(u)\|_{\mathcal{E}}\le M.$$

\bigskip

\noindent
To show the contraction estimate, take $u, v \in \E$. Using Strichartz estimates, we find that
\begin{align*}
&d(\Phi(u)-\Phi(v))\lesssim \| |u|^{\sigma}u- |v|^{\sigma}v\|_{L^{\gamma'}((0,T),L^{\rho'})}.
\end{align*}
Since
$||u|^{\sigma}u- |v|^{\sigma}v|\lesssim(|u|^\sigma-|v|^\sigma)|u-v|$, it follows from Hölder's inequality that
\begin{align*}
&d(\Phi(u)-\Phi(v))\le \|(\| u\|^\sigma_{L^{p}}+\| v\|^\sigma_{L^{p}})\| u-v\|_{L^{r}}\|_{L^{\gamma'}(0,T)}.
\end{align*}
Hence
\begin{align*}
&d(\Phi(u)-\Phi(v))\lesssim T^{\frac{1}{\gamma'}-\frac{1}{q}}(\| u\|^\sigma_{L^{\infty}((0,T),\z^s_{p})}+\| v\|^\sigma_{L^{\infty}((0,T),\z^s_{p})})d( u,v).
\end{align*}
By inequality \eqref{minh}, we infer that
\begin{align}\label{m}
&d(\Phi(u)-\Phi(v))\lesssim T^{\frac{1}{\gamma'}-\frac{1}{q}}(M+T\|u_0\|_{\z_p^s})^\sigma d( u,v).
\end{align}
Therefore, for $T=T(\| u_0\|_{\z^s_{p}})$ small enough, the mapping $\Phi:\E \rightarrow \E$ is a contraction and, by Banach's fixed point theorem, $\Phi$ has a unique fixed point $u$ in $\E$. Since $u=\Phi(u)\in C([0,T], \z_p^s(\re^d))$, the result follows.

\bigskip

\noindent
{\bf Case 2: $2<p<2d(\sigma+1)/(d+2)$ and $d\ge 3$.}\\
Here, we set $r=\rho=d(\sigma+2)/(d+s\sigma)$. Once again, by Proposition 4.9.4 in \cite{cazenave},
$$\||u|^{\sigma}u\|_{B^{s,2}_{\rho'}}\lesssim \|u\|_{L^{a}}^{\sigma}\|u\|_{B^{s,2}_{r}},$$
where $a=d(\sigma+2)/(d-2s)$. Also, the injection
$$B_{\rho}^{s,2}(\re^d)\hookrightarrow W^{s,\rho}(\re^d)\hookrightarrow L^{\rho^*}(\re^d)$$
holds, where $\rho^*=d\rho/(d-s\rho)$. A simple computation shows that $p<a<\rho^*$, hence, by interpolation, $$B_{\rho}^{s,2}(\re^d)\cap L^p(\re^d)\hookrightarrow L^{a}(\re^d).$$
Thus
$$\||u|^{\sigma}u\|_{B^{s,2}_{\rho'}}\lesssim (\|u\|_{L^p}+\|u\|_{B^{s,2}_{r}})^{\sigma}\|u\|_{B^{s,2}_{r}},$$
which compares to \eqref{compare}. A slight modification of the inequalities performed in the first case then yields the result.
\end{proof}
\begin{remark}
	A simple way to understand the restriction $p\le 2\sigma+2$ is to notice that, when one applies a Strichartz estimate on the Duhamel term,
	$$
	\left\|\int_0^t e^{i(t-\tau)\Delta}|u(\tau)|^\sigma u(\tau)d\tau \right\|_{L^q((0,T), L^r)} \lesssim \||u|^\sigma u\|_{L^{\gamma'}((0,T), L^{\rho'})}= \|u\|_{L^{\gamma'(\sigma+1)}((0,T), L^{\rho'(\sigma+1)})}^{\sigma+1},
	$$
	one needs $p\le\rho'(\sigma+1)$ to ensure integrability in space. Since $\rho'\le 2$, the condition follows.
\end{remark}



\section{Local well-posedness for $p> 2\sigma+2$}\label{main2}
\noindent
In this section we treat the remaining case $p> 2\sigma +2$. Note that this restriction implies that
\begin{equation}
\label{sigmarestriction}
\sigma\le \frac p2-1\le \frac{2s}{d-2s}. 
\end{equation}
\begin{proof}[Proof of Theorem \ref{teo:exist_p_grand}]
	Given admissible pairs $(q_j, r_j), j=1,2,3,$ (to be fixed later on), define
	\begin{align}
	\mathcal{Y}\coloneqq  \Bigg\{&u\in C([0,T],\z^s_p(\re^d)) \cap \Big(\bigcap_{j=1}^3L^{q_j}((0,T),\dot{B}^{s,2}_{r_j}(\re^d))\Big) \,:\,\nonumber\\
	&\|u\|_{\mathcal{Y}}\coloneqq \|u\|_{L^\infty((0,T),L^p)}+\sum_{j=1}^3 \|u\|_{L^{q_j}((0,T), \dot{B}^{s,2}_{r_j})}\le M\Bigg\}.
	\end{align}
	endowed with the natural metric $d(u,v)=\|u-v\|_{\mathcal{Y}}$. It is clear that $(\mathcal{Y}, d)$ is a complete metric space.
	
	As in the previous section, we consider the map
\begin{equation}\label{eq:funcY}
\Phi[u](t)=e^{it\Delta}u_0+i\lambda\int_0^te^{i(t-\tau)\Delta}|u(\tau)|^\sigma u(\tau)d\tau
\end{equation}
and prove that $\Phi:\mathcal{Y}\to \mathcal{Y}$ is a contraction, for $M=2\|u_0\|_{\z_p^s}$ and $T$ small. 

\noindent\textit{Step 1.} Taking the $L^p$ norm in \eqref{eq:funcY} and using Theorem \ref{teo:linear}, we find that
\begin{align}
 \|\Phi[u](t)\|_{L^p}&\le\|e^{it\Delta}u_0\|_{L^p}+|\lambda|\int_0^t\|e^{i(t-\tau)\Delta}|u(\tau)|^\sigma u(\tau)\|_{L^p}d\tau\nonumber\\
&\lesssim (1+|t|)\|u_0\|_{\z_s^p}+\int_0^t(1+|t-\tau|)\||u(\tau)|^{\sigma}u(\tau)\|_{\z_s^p}d\tau\nonumber\\
&\lesssim (1+T)\Big(\|u_0\|_{\z_s^p}+\int_0^T\|u(\tau)\|_{L^{p(\sigma+1)}}^{\sigma+1}d\tau+\int_0^T\||u(\tau)|^{\sigma}u(\tau)\|_{\dot{H}^s}d\tau\Big).\label{eq:lp}
\end{align}
By the homogeneous version of \eqref{compare} (see \cite{cazenave}), we know that
$$\||u|^{\sigma}u\|_{\dot{H}^s}=\||u|^{\sigma}u\|_{\dot{B}^{s,2}_2}\lesssim \|u\|_{L^a}^{\sigma}\|u\|_{\dot{B}^{s,2}_{r_1}},$$
where we choose 
$$a=p, \quad r_1=\dfrac{2p}{p-2\sigma}>2.$$Observe that with this choice, we have $r_1<p<2^*$. Therefore
$$\int_0^T\||u(\tau)|^{\sigma}u(\tau)\|_{\dot{H}^s}d\tau\lesssim \int_0^T\|u(\tau)\|_{L^p}^{\sigma}\|u(\tau)\|_{\dot{B}^{s,2}_{r_1}}d\tau\lesssim T^{\frac 1{q_1'}}\|u\|^{\sigma}_{L^{\infty}((0,T),L^p)}\|u\|_{L^{q_1}((0,T),\dot{B}^{s,2}_{r_1})}$$
where the pair $(q_1,r_1)$ is admissible.

\bigskip

\noindent
We now focus on the second term in \eqref{eq:lp}. Observe that, for $r_2>2$ (see \cite[Theorem 6.5.1]{berghlofstrom}),
$$\dot{B}_{r_2}^{s,2}(\re^d)\hookrightarrow \dot{B}_{p_2}^{0,2}(\re^d)\hookrightarrow L^{\rho}(\re^d),\quad\dfrac 1{\rho}=\dfrac 1{r_2}-\dfrac sd. $$
\noindent
Hence, by interpolation,  under the condition
\begin{equation}
\label{4c1}
p(\sigma+1)<\rho=\dfrac{dr_2}{d-sr_2},
\end{equation}
there exists $0<\theta<1$ such that
\begin{equation}
\label{deftheta}
\frac 1{p(\sigma+1)}=\frac{\theta}{p}+\frac{1-\theta}{\rho},
\end{equation}
and we get
$$\|u\|_{L^{p(\sigma+1)}}^{\sigma+1}\lesssim \|u\|_{L^p}^{\theta(\sigma+1)}\|u\|_{L^{\rho}}^{(1-\theta)(\sigma+1)}\lesssim \|u\|_{L^p}^{\theta(\sigma+1)}\|u\|_{\dot{B}_{r_2}^{s,2}}^{(1-\theta)(\sigma+1)}.$$
In conclusion, for $(q_2,r_2)$ admissible,
\begin{align*}
\int_0^T\|u(\tau)\|_{L^{p(\sigma+1)}}^{\sigma+1}d\tau&\lesssim \|u\|_{L^{\infty}((0,T),L^{p})}^{\theta(\sigma+1)}\int_0^T\|u(\tau)\|_{\dot{B}_{r_2}^{s,2}}^{(1-\theta)(\sigma+1)}d\tau\\&\lesssim T^{\frac{q_2-(1-\theta)(\sigma+1)}{q_2}}\|u\|_{L^{\infty}((0,T),L^{p})}^{\theta(\sigma+1)}\|u\|^{{(1-\theta)(\sigma+1)}}_{L^{q_2}((0,T),\dot{B}^{s,2}_{r_2})},
\end{align*}
provided that $q_2> (1-\theta)(\sigma+1)$, that is,
\begin{equation}
\label{4c2}
\frac 1{\sigma}+\frac{sp}{d\sigma}-\frac d4> \frac 1{r_2} \Big(\frac p\sigma-\frac d2\Big).
\end{equation}
We now choose $r_2$ such that conditions \eqref{4c1} and \eqref{4c2} hold:

\bigskip

\noindent
{\bf Case 1: $d=1,2$.}

\noindent
	We take $r_2=(d/s)^-$. Then \eqref{4c1} holds, and, replacing $r_2=d/s$ in \eqref{4c2}, we obtain the equivalent condition $\sigma<4/(d-2s)$, which holds in view of \eqref{sigmarestriction}.

\bigskip

\noindent
{\bf Case 2: $d\geq 3$.}

\noindent
	If $d/s\le 2d/(d-2)$, we can proceed exactly as in the previous case. If not, we choose $r_2=(2d/(d-2))^-$.
	Replacing in \eqref{4c1} we get the condition
	$$\displaystyle \sigma<\frac{2d}{p(d-2-2s)}-1.$$
	Taking the highest value of $p$, $p=2d/(d-2s)$, this condition amounts to
	$\sigma<2/(d-2-2s)$,
	which, in view of $\eqref{sigmarestriction}$, holds for $s<1$.

\bigskip

\bigskip

\noindent
We now examine the condition \eqref{4c2}. We claim that $2p>d\sigma$: if not, then $4p\le 2d\sigma \le d(p-2)$, which implies $d\ge 5$ and $p\ge 2d/(d-4)>2^*$, which is absurd. Taking $r_2=2d/(d-2)$, condition  \eqref{4c2} translates to
$$d(\sigma-2)<p(2s+2-d), $$
which holds, due to \eqref{sigmarestriction} and $s<2$. 

\bigskip

\noindent\textit{Step 2.} Now we deal with the estimates in $L^{q_j}((0,T), \dot{B}^{s,2}_{r_j}(\re^d))$. To do this, we set $r_3=2p/(p-\sigma)$ and let $q_3$ be such that $(q_3,r_3)$ is admissible.  Therefore
$$p=\sigma \frac{r_3}{r_3-2}=\sigma \frac{r_3r_3'}{r_3-r_3'}.$$
Since $p-2\sigma\geq 2$, it follows that $$2<r_3<p<\frac{2d}{d-2s}<\frac{2d}{d-2} \quad \mbox{for}\quad s<1.$$ 
Given any admissible pair $(q,r)$, by Strichartz estimates, we have that
$$\|\phi[u]\|_{L^q((0,T),{\dot{B}}^{s,2}_r)}\lesssim \|u_0\|_{\dot{H}^s}+ \||u|^{\sigma}u\|_{L^{q_3'}((0;T),\dot{B}^{s,2}_{r_3'})}.$$
Then, following the ideas in Step 1, we infer that
\begin{equation}
\label{p1} 
\||u|^{\sigma}u\|_{\dot{B}^{s,2}_{r_3'}}\lesssim \|u\|_{L^{a}}^{\sigma}\|u\|_{\dot{B}^{s,2}_{r_3}}, 
\end{equation}
where $a=\sigma r_3r_3'/(r_3-r_3')=p$.
Finally, by H\"older's inequality, we conclude that
$$ \||u|^{\sigma}u\|_{L^{q'}((0;T),\dot{B}^{s,2}_{r'})}\le \|u\|_{L^{\infty}((0,T);L^{p})}^{\sigma}\Big(\int_0^T\|u(\tau)\|_{\dot{B}_{r_3}^{s,2}}^{q_3'}d\tau\Big)^{\frac{1}{q_3'}}$$
$$\lesssim T^{\frac{1}{q_3'}-\frac {1}q_3}\|u\|_{L^{\infty}((0,T);L^{p})}^{\sigma}\|u\|_{L^{q_3}((0,T),\dot{B}^{s,2}_{r_3})}.$$
Note that this estimate also holds for $(q,r)=(+\infty,2)$, in which case we obtain the control of the $L^{\infty}((0,T),\dot{B}^{s,2}_2(\re^d))\cong L^{\infty}((0,T),\dot{H}^s(\re^d))$ norm.

\bigskip

\noindent\textit{Step 3.} The previous steps imply that, for $T$ small, $\Phi:\mathcal{Y} \to \mathcal{Y}$. As the metric is the one defined by the norm $\|\cdot\|_{\mathcal{Y}}$, the contraction estimate follows by analogous arguments.
\end{proof}

\section{Proof of Theorem \ref{teo:existglobal}}\label{sec:existglobal}

Throughout this section, we suppose $d=3$, $\sigma=2$, $4<p<6$ and $\lambda=-1$. We recall the formal energy functional
$$
E(u)=\frac{1}{2}\int |\nabla u|^2 dx + \frac{1}{4}\int |u|^4 dx.
$$
\begin{proposition}\label{prop:inducao}
	Fix $\epsilon>0$ small and $\theta=(p-4)/(2p-4)$. Given $u_0\in \z_p^1(\re^3)$, let $\phi\in \z^1_4(\re^3)$ and $\psi\in \z^1_p(\re^3) \cap (\bigcap_{\alpha>1} \dot{H}^\alpha(\re^3))$ satisfy
	$$
	u_0=\phi+\psi,\quad E(\phi)\lesssim \epsilon^{-4\theta}, \quad \|\psi\|_{\z^1_p}\lesssim 1,\quad \|\psi\|_{\dot{H}^\alpha}\lesssim \epsilon^{\alpha-1}.
	$$
	Then the solution $u$ of \eqref{NLS} with initial condition $u_0$ is defined on $[0,\delta],\ \delta=\epsilon^{8\theta^+}$. Moreover, there exists $\Phi\in \z^1_4(\re^3)$ such that
	$$
	u(\delta) = \Phi + e^{i\delta\Delta}\psi
	$$
	and
	$$
	E(\Phi)-E(\phi) \lesssim \epsilon^{\beta-3\theta}, \quad \beta=\frac{1+4\theta}{2}.
	$$
\end{proposition}
\begin{proof} Throughout the proof, we abbreviate $L^q((0,\delta), Y)$ to $L^q_tY_x$.
	
\noindent\textit{Step 1.} Let $v$ be the $\z_4^1$-solution of \eqref{NLS} with initial condition $\phi$ (whose existence was proved in \cite{simaoLpH1}). Due to the conservation of energy, $v$ is defined on $[0,\delta]$, $E(v(t))=E(\phi)\lesssim \epsilon^{-4\theta}$ and
$$
	\|v\|_{L^{8/3}_t L^4_x} \lesssim \delta^{3/8}\|v\|_{L^{\infty}_t L^4_x} \lesssim \delta^{3/8} E(\phi)^{1/4}.
$$
Furthermore, letting $(q,r)$ be an admissible pair and then applying Strichartz estimates on the Duhamel formula, we find that
	\begin{align*}
	\|\nabla v\|_{L^q_t L^r_x} &\lesssim \|\nabla \phi\|_{L^2} + \| |v|^2\nabla v\|_{L^{8/5}_tL^{4/3}_x}\\ &\lesssim \|\nabla \phi\|_{L^2} + \|v\|_{L^\infty_t L^4_x}^2\| \nabla v\|_{L^{8/5}_tL^{4}_x} \\ &\lesssim  E(\phi)^{1/2} + \delta^{1/4}E(\phi)^{1/2}\| \nabla v\|_{L^{8/3}_tL^{4}_x}.
	\end{align*}
	Since $\delta^{1/4}E(\phi)^{1/2}=\epsilon^{0^+}$, for $\epsilon$ small, we conclude that $\|\nabla v\|_{L^q_t L^r_x} \lesssim E(\phi)^{1/2}$.
	
	\bigskip

	\noindent\textit{Step 2.} We define the function $w$ by setting
	$$
	u=v + e^{it\Delta}\psi + w.
	$$
	Then $w$ is a solution of
	$$
	w(t)=i\int_0^t e^{i(t-\tau)\Delta}\left( |u(\tau)|^2 u(\tau) - |v(\tau)|^2 v(\tau)  \right)d\tau.
	$$
	Now we estimate $w$ in $L^q_t W^{1,r}_x$ over the time interval $[0,\delta]$, for any $(q,r)$ admissible. Let us first observe that, since the nonlinearity is cubic, one may first split it into several terms and then apply (a possibly different) Strichartz estimate to each one individually. We are then left with estimates in some dual Strichartz space $L^{\gamma'}_t L^{\rho'}_x$. Using Hölder's inequality, the analysis can be reduced to four terms:
	$$
	w^3,\quad v^2w,\quad v^2(e^{it\Delta}\psi),\quad (e^{it\Delta}\psi)^3.
	$$

	\noindent \textbf{Estimate for $w^3$}:
	\begin{align*}
	\||w|^3\|_{L^{8/5}_t W^{1,4/3}_x} \lesssim \delta^{1/4} \|w\|_{L^{\infty}_t H^1_x}^2\|w\|_{L^{8/3}_t W^{1,4}_x}\lesssim \epsilon^{2\theta^+} \|w\|_{L^{\infty}_t H^1_x}^2\|w\|_{L^{8/3}_t W^{1,4}_x}
	\end{align*}
	
	\noindent \textbf{Estimate for $v^2(e^{it\Delta}\psi)$}: for $\rho=2^+$, when no derivatives are present, we estimate
	\begin{align*}
	\||e^{it\Delta}\psi||v|^2\|_{L^{\gamma'}_t L^{\rho'}_x} &= \delta^{1^-}\|e^{it\Delta}\psi\|_{L_t^\infty L_x^{\infty^-}}\|v\|_{L^{\infty}_tL^4_x}^2 \\&\lesssim \delta^{1^-}\|\psi\|_{\dot{H}^{(3/2)^-}}E(\phi)^{1/2} \lesssim \epsilon^{1/2+6\theta}
	\end{align*}
	We now consider the estimates in $L^q_t\dot{W}_x^{1,r}$. On one hand, when the derivative falls onto the free solution,
	\begin{align*}
	\||\nabla e^{it\Delta}\psi||v|^2\|_{L^{\gamma'}_t L^{\rho'}_x} &= \delta^{1^-}\|e^{it\Delta}\nabla\psi_0\|_{L_t^\infty L_x^{\infty^-}}\|v\|_{L^{\infty}_tL^4_x}^2 \\&\lesssim \delta^{1^-}\|\psi\|_{\dot{H}^{(5/2)^-}}E(\phi)^{1/2} \lesssim \epsilon^{ 3/2+6\theta}
	\end{align*}
	On the other hand, if the derivative falls onto $v$,
	\begin{align*}
	\||e^{it\Delta}\psi||v||\nabla v|\|_{L^{\gamma'}_t L^{\rho'}_x} &= \delta^{5/8^-}\|e^{it\Delta}\psi\|_{L_t^\infty L_x^{\infty^-}}\|v\|_{L^{\infty}_tL^4_x}\|\nabla v\|_{L^{8/3}_tL^4_x} \\&\lesssim \delta^{5/8^-}\|\psi\|_{\dot{H}^{(3/2)^-}}E(\phi)^{3/4} \lesssim \epsilon^{1/2+ \theta}
	\end{align*}
	
	\noindent \textbf{Estimate for $(e^{it\Delta}\psi)^3$}: for $\rho=2^+$,
	\begin{align*}
	\| |e^{it\Delta}\psi|^{3} \|_{L^{\gamma'}_t L^{\rho'}_x} \lesssim \| e^{it\Delta} \psi \|_{L^{3\gamma'}_t L^{3\rho'}_x}^3 \lesssim \delta^{1^-}\left( \|e^{it\Delta}\psi\|_{L^p}^{(p/6)^+}\|e^{it\Delta}\psi\|_{\dot{H}^{(3/2)^-}}^{1-(p/6)^+}\right)^3\lesssim \epsilon^{8\theta+(6-p)/4}
	\end{align*}
	For the term with a derivative, the estimate is improved:
	\begin{align*}
	\| |e^{it\Delta}\psi|^{2}\nabla e^{it\Delta}\psi \|_{L^{\gamma'}_t L^{\rho'}_x} &\lesssim \| e^{it\Delta} \psi \|_{L^{3\gamma'}_t L^{3\rho'}_x}^2\| \nabla e^{it\Delta} \psi \|_{L^{3\gamma'}_t L^{3\rho'}_x} \\&\lesssim \delta^{1^-}\epsilon^{(6-p)/6}\left( \|\nabla e^{it\Delta}\psi\|_{L^2}^{(1/3)^+}\|\nabla e^{it\Delta}\psi\|_{\dot{H}^{(3/2)^-}}^{1-(1/3)^+}\right) \lesssim \epsilon^{8\theta+(6-p)/6+1}
	\end{align*}
	
	\noindent\textbf{Estimate for $v^2 w$}: once again, we estimate separately the term, depending on whether it has no derivatives, a derivative on $w$ or a derivative on $v$:
	\begin{align*}
	\|v^2 w\|_{L^{8/5}_t L^{4/3}_x} \lesssim \|v\|_{L^\infty_t L^4_x}^2 \| w\|_{L^{8/5}_t L^4_x} \lesssim E^{1/2}\delta^{1/4}\|w\|_{L^{8/3}_t L^4_x}\lesssim \epsilon^{0^+}\|w\|_{L^{8/3}_t L^4_x}
	\end{align*}
	\begin{align*}
	\|v^2 \nabla w\|_{L^{8/5}_t L^{4/3}_x} \lesssim \|v\|_{L^\infty_t L^4_x}^2 \| \nabla w\|_{L^{8/5}_t L^4_x} \lesssim E^{1/2}\delta^{1/4}\|\nabla w\|_{L^{8/3}_t L^4_x}\lesssim \epsilon^{0^+}\|\nabla w\|_{L^{8/3}_t L^4_x}
	\end{align*}
	\begin{align*}
	\|vw\nabla v\|_{L^2_tL^{6/5}_x}&\lesssim \delta^{1/2}\left\|\|v\|_{L_x^6}\|w\|_{L_x^6}\|\nabla v\|_{L_x^2} \right\|_{L^\infty_t} \lesssim \left\|\|\nabla v\|_{L_x^2}\|\nabla w\|_{L_x^2}\|\nabla v\|_{L_x^2} \right\|_{L^\infty_t}\\&\lesssim \delta^{1/2}E\|w\|_{L^\infty_t H^1_x} \lesssim \epsilon^{0^+}\|w\|_{L^\infty_t H^1_x} .
	\end{align*}
	In conclusion, setting $X=\underset{(q,r)}{\sup}\|w\|_{L^q_tW^{1,r}_x}$,
	\begin{align*}
	X\lesssim \epsilon^{1/2 + \theta} + \epsilon^{0^+}X + \epsilon^{2\theta}X^3.
	\end{align*}
	By obstruction, we conclude that 
	$$
	\underset{(q,r)}{\sup}\|w\|_{L^q_tW^{1,r}_x} \lesssim \epsilon^{1/2 + \theta} = \epsilon^\beta.
	$$
	\textit{Step 3.} Define
	$$
	\Phi= v(\delta) + w(\delta).
	$$
	Then $\Phi\in \z_4^1$, $u(\delta)=\Phi + e^{i\delta\Delta}\psi$ and
	\begin{align*}
	E(\Phi)-E(\phi)&= (E(v(\delta) + w(\delta))- E(v(\delta)))\\&\lesssim  \|\nabla w(\delta)\|_{L^2}\left(\|\nabla v(\delta)\|_{L^2} + \|\nabla w(\delta)\|_{L^2}\right) + \|w(\delta)\|_{L^4}\left( \|v(\delta)\|_{L^4}^3 + \| w(\delta)\|_{L^4}^3  \right)\\&\lesssim \epsilon^{\beta}E(\phi)^{1/2} +  \epsilon^{2\beta} + \epsilon^\beta E(\phi)^{3/4} +\epsilon^{4\beta} \lesssim \epsilon^{\beta-3\theta}.
	\end{align*} 
\end{proof}

\begin{proof}[Proof of Theorem \ref{teo:existglobal}]
It suffices to prove that any initial condition $u_0\in \z_p^1(\re^3)$ gives rise to a solution $u$ defined on $[0,1]$. In this interval, the linear group is bounded in $\z_p^1(\re^3)$ uniformly in time.

\noindent\textit{Step 1.} Fix a cut-off function supported on the unit ball $\chi$ and set $\chi_\epsilon(\xi)=\chi(\xi/\epsilon)$. Define
$$
\psi = \check{\chi}_\epsilon \ast u_0\quad \mbox{and}\quad \phi=u_0-\psi.
$$
Then, for any $\epsilon>0$ small and $\alpha>1$, using Mikhlin's multiplier theorem,
	$$
	\|\psi\|_{\z^1_p}, \|\phi\|_{\z^1_p}\lesssim \|u_0\|_{\z^1_p}.
	$$
	Observe that these inequalities are uniform in $\epsilon$ (see \cite[Theorem 6.1.3]{berghlofstrom}). Moreover, 
	$$
	\|\psi\|_{\dot{H}^\alpha} \lesssim \epsilon^{\alpha-1} \|u_0\|_{\z^1_p} \quad\mbox{and}\quad \|\phi\|_{L^2}\lesssim \epsilon^{-1}\|\phi\|_{\dot{H}^1} \lesssim \epsilon^{-1}\|u_0\|_{\z^1_p}.
	$$
	By interpolation, we conclude that, for $\theta=(p-4)/(2p-4)$,
	$
	\|\phi\|_{L^4}\lesssim \epsilon^{-\theta}\|u_0\|_{\z_p^1}
	$
	and $E(\phi)\lesssim  \epsilon^{-4\theta}$.

\noindent \textit{Step 2.} Applying Proposition \ref{prop:inducao} and setting $\delta=\epsilon^{8\theta^+}$, $u$ is defined on $[0,\delta]$ and, at time $t=\delta$, it can be decomposed as $\Phi+e^{i\delta\Delta}\psi$. Since $\Phi$ and $e^{i\delta\Delta}\psi$ are still in the conditions of Proposition \ref{prop:inducao}, the result may be applied once more to conclude that $u$ is defined on $[0,2\delta]$. The argument may be iterated for as long as the energy $E(\phi)$ satisfies the bound $E(\phi)\lesssim \epsilon^{4\theta}$. Since the number of steps needed to reach $T=1$ is of the order $\delta^{-1}$, we require
$$
\mbox{Energy increment}\times\mbox{Number of steps}\sim \epsilon^{\beta-3\theta} \delta^{-1} \ll \epsilon^{-4\theta},
$$
that is, $\beta>7\theta$. This condition is verified whenever $4<p<9/2$. The proof is finished.
\end{proof}

\section{Acknowledgements}
V. Barros was partially supported by CNPq, by FCT project PTDC/MAT-PUR/28177/2017, with national funds, and by CMUP (UIDB/00144/2020), which is funded by FCT with national (MCTES) and European structural funds through the programs FEDER, under the partnership agreement PT2020. S. Correia  was partially supported by Funda\c{c}\~ao para a Ci\^encia e Tecnologia, through the grant UID/MAT/04459/2019. F. Oliveira was partially supported by Funda\c{c}\~ao para a Ci\^encia e Tecnologia, through the grant UID/MULTI/00491/2019.

\bibliography{Biblioteca}
\bibliographystyle{plain}
\begin{center}
	{\scshape Vanessa Barros}\\
	{\footnotesize
	Departamento de Matemática, Universidade Federal da Bahia\\ Av. Ademar de Barros s/n, 40170-110 Salvador, Brazil\\
	and\\
	UniLaSalle, Campus de Ker Lann, 35170 Bruz, France\\
	vanessa.barros@unilasalle.fr
	}
	\vskip15pt
	{\scshape Simão Correia}\\
	{\footnotesize
		Center for Mathematical Analysis, Geometry and Dynamical Systems,\\
		Department of Mathematics,\\
		Instituto Superior T\'ecnico, Universidade de Lisboa\\
		Av. Rovisco Pais, 1049-001 Lisboa, Portugal\\
		simao.f.correia@tecnico.ulisboa.pt
	}
	\vskip15pt
	{\scshape Filipe Oliveira}\\
	{\footnotesize
		Mathematics Department and CEMAPRE\\
		ISEG, Universidade de Lisboa,\\
		Rua do Quelhas 6, 1200-781 Lisboa, Portugal\\
		foliveira@iseg.ulisboa.pt
	}

\end{center}
\end{document}